\documentclass{amsart}

\usepackage[utf8]{inputenc}
\usepackage[english]{babel}

\usepackage[a4paper]{geometry}

\usepackage{amsmath}
\usepackage{amssymb}
\usepackage{amsfonts}
\usepackage{amsthm}
\usepackage{color}
\usepackage{graphicx}
\usepackage{mathtools}
\usepackage{cases}
\usepackage{hyperref}

\newtheorem{theorem}{Theorem}[section]
\newtheorem{corollary}[theorem]{Corollary}
\newtheorem{proposition}[theorem]{Proposition}
\newtheorem{lemma}[theorem]{Lemma}

\theoremstyle{definition}
\newtheorem{definition}{Definition}[section]

\theoremstyle{remark}
\newtheorem{remark}{Remark}

\newcommand{\R}{\mathbb{R}}
\newcommand{\C}{\mathbb{C}}
\newcommand{\N}{\mathbb{N}}

\newcommand{\V}{\textsc{vdm}}
\newcommand{\Vm}{V}

\begin{document}
\title{ Multidimensional pseudo Leja sequences }
\author{Dimitri Jordan Kenne}
\address{Doctoral school of Exact and Natural Sciences, Jagiellonian University, Łojasiewicza 11, Kraków, 30-348, Lesser Poland, Poland.}
\email{dimitri.kenne@doctoral.uj.edu.pl}
\date{December 2023}

\begin{abstract}
The one-dimensional pseudo Leja sequences introduced in \cite{bialas2012pseudo}, as an alternative to Leja sequences, provides us with good interpolation nodes for the approximation of holomorphic functions. We propose a definition of multidimensional pseudo Leja sequences associated to a compact set $K$ of the complex space $\mathbb{C}^p$ which generalises both the one-dimensional version and the multidimensional Leja sequences. We show that these sequences can be used to calculate the transfinite diameter of $K$. We also present a relation to the pluricomplex Green function associated to $K$. Subsequently, we show that the intertwinning of pseudo Leja sequences is still a pseudo Leja sequence. We give a method to compute pseudo Leja sequences with the help of discrete meshes.
\end{abstract}



\maketitle

\section{Introduction} \label{sec1}
We are interested in the computation of goods points for approximation by multivariate polynomial interpolation. Here, by good points, we mean arrays of interpolation nodes in a regular ( see below) compact set $K$ for which the corresponding sequence of Lagrange interpolation polynomials of any holomorphic function $f$ on a neighbourhood of $K$, converges uniformly to $f$ on $K$. In the univariate case, a characterisation of such points is very well-know (\cite{bloom2012polynomial}, \cite{bloom1992polynomial}, \cite{gaier1987lectures}). In contrast, in the multivariate case, no characterisation is known and very few examples of such good points are currently available. The only general example is given by arrays of Fekete points (Fekete points for $K\subset \mathbb{C}^p$ are points which maximises the (modulus) of the Vandermonde determinant of a given order, see below). Unfortunately, these points are impossible to determine explicitly and even the numerical computation of their discrete version is complicated, so that they are essentially of theoretical interest. Apart from the Fekete points, the sole examples of good points are furnished by the intertwining of univariate good points (\cite{bialas2012pseudo}, \cite{bertrand2022newton}, \cite{siciak1962some}) and special remarkable configurations such as the famous Padua points in the unit square of $\R^2$ ( \cite{bos2006bivariate}, \cite{bos2007bivariate}, \cite{caliari2005bivariate}). The interest of other configurations is supported by numerical experiments (see e.g. \cite{baglama1998fast}, \cite{bos2011multivariate}). 
\par\smallskip 
In the univariate case, while the nature of good points is completely understood from the theoretical point of view, the computation of such points for an arbitrary regular compact set is far from being easy. A way of constructing univariate good points were introduced in \cite{bialas2012pseudo} with the notion of pseudo Leja sequences. Such a sequence is defined recursively as follows. The first point $\xi_0 \in K$ is chosen arbitrarily and, if the first $j$ points $\xi_0, \dots, \xi_{j-1}$ are already obtained, then $\xi_j$ is chosen in $K$ to ensure that   
\begin{equation}
M_j \prod_{i=0}^{j-1} |\xi_j -\xi_i|  \ge \max_{z\in K}\prod_{i=0}^{j-1} |z -\xi_i|. \label{equ:0}
\end{equation}
Here $(M_j)$ is a sequence of real numbers $M_j \ge 1, \ j \ge 1$ which is required to be of sub-exponential growth.
The classical Leja points correspond to the special case, $M_j=1$, $j \ge 1$. The point is that the possibility $M_j>1$ alleviates the computation of the points without weakening the approximation quality of the Leja points : it is shown in \cite[Theorem 1, page 56]{bialas2012pseudo} that they are good points (for reasonable compact sets) in the sense above. The strong result that, on a set of positive logarithmic capacity, the Lebesgue constants of one-dimensional pseudo Leja points with bounded sequence $(M_j)$ are subexponential has been recently established in \cite{totik2023lebesgue}.  
\par\smallskip
It is readily seen that the product in \eqref{equ:0} can be replaced by the (modulus) of a Vandermonde determinant and this gives the way to the extension of the definition to the higher dimensional case. The study of this generalisation  is the subject of the present paper. The precise definition of a multivariate pseudo Leja sequence is given in Section \ref{sec2} . Although we are not able to show that our multivariate pseudo Leja points are good points (in the multivariate case, this is not even known for ordinary Leja sequence), we prove that they permit to recapture the multivariate transfinite diameter of $K$, which, in turn, by deep results due to Berman, Boucksom and Nyström (\cite[Theorem A, page 3 ]{berman2011fekete}) and Bloom (\cite[Corollary 4.5, page 1563]{bloom1998distribution}), implies that they permit to recover the multivariate equilibrium measure and the pluri-complex Green function for a regular compact set $K$. This is usually regarded as an indicator that the points should actually be good in the above sense, especially because these conditions are equivalent in the univariate case (see \cite[Theorem 1, page 65]{gaier1987lectures} or \cite[Theorem 1.5, page 446]{bloom1992polynomial}).
\par\smallskip
In  Section\ref{sec:computing}, we will briefly explain how to compute pseudo-Leja points, showing that discrete Leja points can be seen as examples of pseudo-Leja points. 
\par\smallskip 
We will also prove an interesting stability property of multivariate pseudo Leja sequences. Roughly speaking, by intertwining, in some specific way, the points of a pseudo Leja sequence for $K_1\in \C^{p_1}$ and that of a pseudo Leja sequence for $K_2\in \C^{p_2}$, we obtain a pseudo Leja sequence for the product  $K_1\times K_2$, see \ref{intertwinning}.
\par\smallskip
For the convenience of the reader, we will now recall standard notations and basic definitions on multivariate interpolation. Multi-indices will be ordered according to the graded lexicographic order, so that $\alpha\prec \beta$ if $|\alpha|<|\beta|$ or $|\alpha|=|\beta|$ and the left-most non zero element of $\alpha -\beta$ is negative. We arrange the elements of $\N^p$ in an increasing sequence $\kappa=\kappa^{(p)}$ (with respect to $\prec$), 
\begin{equation}\label{eq:kappa}
\kappa=\kappa^{(p)} : \N \to \N^p. 
\end{equation}
Thus, when $p=2$, we have
\begin{multline*}
\kappa(0)=(0,0)\prec \kappa(1)=(0,1) \prec \kappa(2)=(1,0) \prec \kappa(3)=(0,2) \\ \prec \kappa(4)=(1,1) \prec \kappa(5)=(2,0) 
\prec \kappa(6)=(0,3) \prec \dots. 
\end{multline*}
The monomials of $p$ complex variables will be denoted as follows
\begin{equation}
	e_N(z)= z^{\kappa(N)} = z_1^{\alpha_1}\cdots z_p^{\alpha_p},\quad \kappa(N)=(\alpha_1,\dots,\alpha_p). \label{equ:12}
\end{equation}
We will use the space of polynomials spanned by the first $N$ monomials :
\begin{equation*}
\mathcal{P}_N = \textbf{span}\{e_i: \ i= 0,\dots, N-1\}, \quad N\ge 1. 
\end{equation*}
The dimension of the classical space of polynomial of total degree not bigger than $d$ is of dimension
\[h_d = \binom{p+d}{d},\]
thus
\[N=h_d-1 \implies \mathcal{P}_N=\textbf{span}\{z\to z^\alpha\;:\; |\alpha|\leq d\}.\]
The Vandermonde determinant of a set of the points $\xi_0, \dots, \xi_q$ in $\C^p$ is defined as 
\begin{equation}
	\V(\xi_0, \dots, \xi_q) = \det\left[e_N(\xi_M)\right]_{0 \leq M,N \leq q}. \label{equ:13}
\end{equation} 
In particular, $\V(\xi_0) = 1$.
Observe that its modulus is independent of the ordering of the points. The Vandermonde determinant is a polynomial in the coordinates of the points $\xi_j$, thus a polynomial in $p(q+1)$ variables. As such, its degree for $q =h_d-1$ is given by  \begin{equation}\label{eq:ld}l_d = \sum_{i=1}^d i (h_i - h_{i-1})=p\binom{p+d}{p+1}.\end{equation}
\par \smallskip
Given a set of $N$ points, $\Omega_N = \{\zeta_{N0}, \dots, \zeta_{NN}\} \subset \C^p$, with a non-zero Vandermonde determinant, we can form the \textbf{Fundamental Lagrange Interpolation Polynomials} (FLIP)
     \begin{equation}\label{eq:flip}
         l_j^{(N)}(z) = \dfrac{\V(\zeta_{N0}, \dots, \zeta_{N(j-1)}, z,\zeta_{N(j+1)} \dots, \zeta_{NN}) }{\V(\zeta_{N0}, \dots, \zeta_{NN}) }, \quad j = 0,\dots, N.
     \end{equation}
Observe that $l_j^{(N)}$ is an element of $\mathcal{P}_{N+1}$. For a function $f$ defined at the points in $\Omega_N$, 
     \begin{equation}\label{eq:lif}
         L_{\Omega_N}f(z) = \sum_{j=0}^{N} f(\zeta_{Nj}) l_j^{(N)}(z)
     \end{equation}
is the \textbf{Lagrange Interpolation Polynomial} (LIP) of $f$ in $\mathcal{P}_{N+1}$ at the points of $\Omega_N$. In particular, 
\begin{equation}f\in \mathcal{P}_{N+1} \implies f= L_{\Omega_N}f(z). \label{eq:lifpol}\end{equation}
In the whole paper the compact sets $K$ considered will be assumed to be (polynomially) \textbf{determining}. This means that no nonzero polynomial vanishes on $K$. Equivalently, $K$ is not included in an algebraic set. This property can also be expressed in terms of Vandermonde determinants. Namely, setting
\begin{equation} \label{eq:defVjK}
V_j(K) = \max_{\xi_0, \dots, \xi_{j-1}\in K} |\V(\xi_0, \dots, \xi_{j-1})|, \quad j\ge 1,    
\end{equation}
then $K$ is determining if and only if no $V_j(K)$ vanishes. The sufficiency follows from an application of \eqref{eq:lifpol}. To see the converse, suppose $V_{N-1}(K) \not = 0$ for $1 \le j < N$ but $V_N(K)=0$ and let $\{\xi_0, \dots \xi_{N-1}\}\subset K$ such that $\V(\xi_0, \dots \xi_{N-1}) \not = 0$. The polynomial $Q_N(z)  = \V(\xi_0, \dots, \xi_{N-1},z)$ vanishes on $K$ but is not the zero polynomial since its $z^{\kappa(N)}$ coefficient is not null so that $K$ is not determining. 
\par \smallskip
If it is necessary to specify the dimension of the ambient space, we will add a superscript to the above objects thus writing for instance $k^{(p)}(N)$, $\mathcal{P}_N^{(p)}$, and so on.
\par \smallskip
For survey on the actual state of knowledge on multivariate polynomial interpolation we refer to \cite{bloom2012polynomial}. 
\section{Pseudo Leja sequences}\label{sec2}	
\begin{definition}\label{def:1} Let $K$ be a determining compact subset in $\C^p$. 
On says that $(\xi_j)$ is a \textbf{pseudo Leja sequence} in $K$ if there exists a sequence of real numbers $(M_j)_{j \ge 1} \subset [1, +\infty)$ satisfying:
\begin{equation}
\label{property1}M_{j} |\V(\xi_0, \dots, \xi_{j-1},\xi_{j})| \ge \max_{z \in K} |\V(\xi_0, \dots, \xi_{j-1}, z)|\quad\text{for any $j  \ge 1$,} \end{equation}
with the growth condition,
\begin{equation}
\label{property2} 
\lim_{d \to + \infty} (\max_{h_{d-1} \le j < h_d} M_j)^{1/d} = 1. 
\end{equation}
We say that $\mathcal{L}$ is a pseudo-Leja sequence of Edrei growth $(M_j)_{j \ge 1}$.
\end{definition}
In the case $p=1$, this definition coincides with that given in \cite{bialas2012pseudo}. A Leja sequence is a pseudo Leja sequence of Edrei growth $1$. The fact that $K$ is determining immediately implies the existence of pseudo Leja sequence with non vanishing Vandermonde determinant. Observe that we may with very limited loss of generality, assume that $(M_j)$ is non-decreasing and, in that case, \eqref{property2} reduces to a somewhat less clumsy condition.
\par 
Observe that if $(\xi_j)$ is a pseudo Leja sequence for $K$ then 
\begin{equation}\label{eq:pslejunis}
    |\V(\xi_0, \dots, \xi_{j-1},\xi_{j})| > 0, \quad \text{for all $j$}.
\end{equation}
This follows from the same reasoning sketched in the introduction : if $j$ was the smallest integer for which
$|\V(\xi_0, \dots, \xi_{j-1}, z)|=0$ on $K$ then 
$P(z)=\V(\xi_0, \dots, \xi_{j-1}, z)$ 
would be a non zero polynomial vanishing on $K$ thus contradicting the fact that it is determining. The non vanishing of the Vandermonde determinant, implies, see \eqref{eq:lif}, the following proposition. 
\begin{proposition} Let $(\xi_i)$ a pseudo Leja sequence as above. Lagrange interpolation at $\Omega_N= \{\xi_{0}, \dots, \xi_{N}\}$, $N\in \N$, is always possible. 
 \end{proposition}
We immediately draw of a consequence of the growth condition that will be used in the proof of our main theorem below.
\begin{lemma}\label{lem:3}
If $\{M_j\}_{j \ge 1}$ is a sequence of real numbers satisfying Property \eqref{property2} from the definition above, then
	\[
		\lim_{d \to + \infty}\big[\prod_{j=1}^{h_d-1} M_{j}\big]^{1/l_d} =1,
	\]
where $l_d$ is defined in \eqref{eq:ld}.
\end{lemma}
\begin{proof} Notice that
\[1 \le \prod_{j =1}^{h_d-1} M_{j} \le (\max_{j =1,\dots,h_d-1}M_j)^{h_d}\quad\text{and}\quad
\frac{h_d}{l_d} = \frac{1}{d} \left(\frac{p+1}{p}\right).\]
Choose the smallest $j(d)$ such that 
\[\max_{j=1,\dots,h_d-1}M_j = \max_{h_{j(d)-1} \le j < h_{j(d)}} M_j.\]
Since it is a non-decreasing sequence of integers, $(j(d))_{d \ge 1}$ tends to $+\infty$ or is eventually constant.
We examine both cases. \par
If $j(d) \to \infty$, then, by \eqref{property2},
\[\lim_{d \to + \infty} (\max_{h_{j(d)-1}\le j < h_{j(d)}} M_j)^{1/j(d)} = 1,\]
hence, since $0 < j(d) \le d $,
    \[\lim_{d \to + \infty} (\max_{h_{j(d)-1} \le j < h_{j(d)}} M_j)^{1/d} = 1.\]
If $j(d)$ is eventually constant, the conclusion is even simpler since $(M_j)$ is then bounded and $h_d/l_d\to 0$.
\end{proof}
We now state our main theoretical result and a consequence. A similar result was obtained in  \cite{jkedrzejowski1992transfinite}  and  \cite{bloom1992polynomial} for Leja sequences. The method of proving Theorem \ref{thm:gettransdiam} is in fact not different from that of Leja sequences except that here we take into consideration the Edrei growth of the pseudo Leja sequence. Its proof is postponed to the next section where we will recall the technical points that we will need. 
\begin{theorem}\label{thm:gettransdiam}
Let $K \subset \C^p$ be a determining compact set. If $(\xi_j)_{j \ge 0} \subset K$ is a  pseudo-Leja sequence of Edrei growth $(M_i)_{j \ge 1}$, then 
	\begin{equation}\label{equ:11}
		\lim_{d \to + \infty} |\V(\xi_0, \dots, \xi_{h_d-1})|^{1/l_d} = D(K),
	\end{equation}
 where $D(K)$ denotes the multivariate transfinite diameter. 
\end{theorem}
Thanks to R. Berman and S. Boucksom outstanding result which first appeared in the preprint \cite[Theorem 1.1]{berman2008equidistribution} and later was presented as a consequence of a more general convergence result in  \cite[Theorem A]{berman2011fekete} with D. W. Nyström, we can derive the following.
\begin{theorem}\label{thm:equi_distribution}
    Let $K \subset \C^p$ be a determining compact set. If $(\xi_j)_{j \ge 0} \subset K$ is a  pseudo-Leja sequence, then the normalised counting measures $h_d^{-1} \sum_{j=0}^{h_d} [\xi_j]$ converge to the equilibrium measure of $K$ in the weak-* sense.
\end{theorem}
Here $[\xi]$ refers to the Dirac measure at $\xi$ of total mass $1$.\\
Setting 
	\begin{equation}
		P_k(\xi)= \dfrac{ \V(\xi_0, \dots, \xi_{k-1}, \xi)}{\V(\xi_0, \dots, \xi_{k-1})}, \label{eq:defPk}
	\end{equation}
we will see that 
\begin{equation}
\lim_{d\to\infty} \big(\prod_{k =0}^{h_d-1} \|P_k\|_K \big)^{1/l_d} = D(K), \label{eq:asympk}
\end{equation} 
and, by a strong result of Bloom (\cite[Corollary 4.5]{bloom2001families}), we will obtain the following. 
\begin{theorem} \label{th:greenfunc}
Let $K$ be a polynomially convex regular compact subset of $\C^p$. Then
	\begin{equation}
		G_K(z) = \limsup_{N \to + \infty} \frac{1}{|\kappa(N)|} \log \left(\dfrac{|P_N(z)|}{\|P_N\|_K}\right), \quad z \in \C^p\setminus K,
	\end{equation}
where $G_K$ denotes the pluri-complex Green function of $K$.
\end{theorem}
For the notions of pluricomplex potential theory used  in the above statement (regularity, Green function, equilibrium measure), we refer the reader to \cite{klimek1991pluripotential}.
\par Observe that in the case where $K$ would be non determining, a pseudo Leja sequence would be eventually arbitrary (as soon as $V_j(K)=0$), and Theorem \ref{thm:gettransdiam} would be trivially true (with $D(K)=0$).   
\section{Transfinite diameter and proof of Theorem \ref{thm:gettransdiam}}\label{sec3}
We will recall in details the objects and results related to the transfinite diameter. We follow the presentation given in \cite{bloom1992polynomial}. As usual, $K$ denotes a determining compact set in $\C^p$. We will write $V_j$ instead of $V_j(K)$, see \eqref{eq:defVjK}, when there is no ambiguity. We define the $d$-th order transfinite diameter of $K$, denoted by $D_d(K)=D_d$, as
\[D_d =  V_{h_d}^{1/l_d}.\]
Fekete proved in \cite{fekete1923verteilung} that the limit (as $d\to \infty)$ of the sequence $D_d$, denoted as $D(K)$, exists for any compact set $K \subset \C$.  Later in \cite{leja1959problemes}, Leja introduced the name \textit{transfinite diameter} and posed the problem of its existence in the multivariate case. A positive answer to his problem was given by Zaharjuta in the celebrated paper \cite{zaharjuta1975transfinite} where it also shown a remarkable connection with some certain polynomials of minimal norm that we will now describe. 
\par
A monic polynomial is a polynomial whose leading term with respect to $\prec$ is equal to $1$. Thus a monic polynomial of leading term $e_i$ is of the form
	\begin{equation}\label{eq:defPfor Tau}
		P(z)= e_i(z) + \sum_{0 \le j < i} c_j e_j(z): \ c_j \in \C\}, \quad i \in \N.
	\end{equation}
We will indicate such a form in writing $\textbf{lead}(P)=e_i$.  
We define the \textbf{$i$-th Chebyshev constant} $\tau_i$ of $K$ as
\begin{equation}\label{eq:taui}
\tau_i= \inf \left\{\|P\|_K^{1/|\kappa(i)|}\;:\;\textbf{lead}(P)=e_i \right\}.
\end{equation}
Given $\theta = (\theta_1, \dots, \theta_p)$ in the standard $p$-simplex $\Sigma$,
\[\Sigma = \big\{ \theta \in \R^n\,:\, \sum_{i=1}^p \theta_i = 1,\; \theta_i \ge 0,\; i = 1, \dots, p \big\},\]
	the limit \begin{equation}
				\tau(K, \theta) = \limsup_{j \to + \infty, \frac{\kappa(j)}{|\kappa(j)|} \to \theta} \tau_j
			\end{equation}
is called the \textbf{directional Chebyshev constant} of $K$ in the $\theta$-direction.
Now, if we denote by $\mathcal{T}_d$, the geometric mean of Chebyshev constants for a given degree,
		\begin{equation}\label{eq:defTi}
			\mathcal{T}_d = \big[ \prod_{|\kappa(i)|= d} \tau_i\big]^{1/(h_d - h_{d-1})},
		\end{equation}
the fundamental result of Zaharjuta can be stated as follows.
\begin{theorem}\cite[Theorem 1]{zaharjuta1975transfinite} \label{th:zaha}The sequence $\mathcal{T}_d$ converges as $d\to \infty$ and we have
		\begin{equation}
			\lim_{d \to \infty} \mathcal{T}_d = \exp\left[\frac{1}{meas(\Sigma)} \int_\Sigma \log(\tau(K, \theta))\ d\theta\right] = D(K)
		\end{equation}
  where $\text{meas}(\Sigma)$ denotes the standard volume of the simplex. 
\end{theorem}
We will need a ``mean modification" of the above. 
\begin{corollary}\label{th:meanmod}
We have \begin{equation}
			\lim_{d\to\infty} \big[\prod_{j =1}^d \mathcal{T}_j^{r_j}\big]^{1/l_d} = D(K) \label{equ:3.17}
		\end{equation} where $r_j = j(h_j - h_{j-1})$ for $j\ge 1$.
\end{corollary}
The corollary follows from Theorem \ref{th:zaha} by applying the following classical lemma of Cesaro type (for which we omit the proof) with $u_d = \log(\mathcal{T}_d)$ and $\theta_{d,j} = r_j/l_d$.
\begin{lemma}
Let $(u_d)$ be a sequence of real numbers converging to $A$. Consider an array of non negative
numbers $\theta_{d,j}$ , $d,j \in \N^\star$ verifying \[ \lim_{d \to + \infty} \max_{j = 1. \dots, d}\theta_{d,j} =0\quad\text{and}\quad \sum_{j=1}^d \theta_{d,j} = 1, \; d\in \N.\] Then the sequence \[s_d = \sum_{j =1}^d \theta_{d,j} u_j,\ d \in \N,\] also converges to $A$.
\end{lemma}
We have now assembled the material required in our proof of Theorem \ref{thm:gettransdiam}.
\begin{proof}[Proof of Theorem \ref{thm:gettransdiam}]
Set $L_k = |\V(\xi_0, \dots, \xi_{k-1})|$ for all
$k \ge 1$. In view of \eqref{eq:defVjK} and \eqref{eq:pslejunis}, we have 
\begin{equation}\label{eq:lksmallerVk} 0< L_k \le V_k, \quad  k\in \N.\end{equation} 
Recall that the polynomial $P_k$ defined in \eqref{eq:defPk} is given by
\begin{equation}
P_k(\xi)=\dfrac{ \V(\xi_0, \dots, \xi_{k-1}, \xi)}{\V(\xi_0, \dots, \xi_{k-1})} = e_{k}(\xi) + \sum_{0\le j< k} c_je_j(\xi), 
	\end{equation}
 Observe that
 \[P_k\in \mathcal{P}^k, \quad \textbf{lead}{P_k}=e_k, \quad \deg(P_k)=\deg(e_k)=\kappa(k).\]
By definition of a pseudo Leja sequence, we have
 \[|P_k(\xi_{k})| = \dfrac{L_{k+1}}{L_{k}} \ge  \dfrac{M_{k}^{-1} \max_{\xi\in K} \V(\xi_0, \dots, \xi_{k-1}, \xi)}{ \V(\xi_0, \dots, \xi_{k-1})}\ge M_{k}^{-1}\|P_k\|_K.\]
Hence by definition of $\tau_k$, see \eqref{eq:taui}, we have 
\begin{equation}
\frac{L_{k+1}}{L_{k}} \ge M_{k}^{-1} \|P_k\|_K \ge M_{k}^{-1} \tau_k^{|\kappa(k)|}. \label{equ7}
\end{equation}
By using \eqref{eq:lksmallerVk} and iterating the above inequality, we deduce that, for $d \ge 0$ 
	\begin{align}
		V_{h_d} &\ge L_{h_d} = \frac{L_{h_d}}{L_{h_{d}-1}} \times \frac{L_{h_{d}-1}}{L_{h_{d}-2}} 
  \times \dots \times \frac{L_2}{L_1} \\
		&\ge \Big[\prod_{j=1}^{h_d-1}M_{j}\Big]^{-1} \Big[\prod_{j=1}^{h_d-1} \tau_j^{|\kappa(j)|}\Big] \label{equ:1}.
	\end{align}
We may re-arrange the terms in the last product above taking into account the definition of $\mathcal{T}_i$ in \eqref{eq:defTi} as follows : 
\begin{equation*}
    \prod_{j=1}^{h_d-1}\tau_j^{[\kappa(j)|} = \prod_{i =1}^d \Big[\prod_{j = h_{i-1}}^{h_i-1} \tau_j \Big]^i = \prod_{i =1}^d\mathcal{T}_i^{i(h_i-h_{i-1})}.
\end{equation*}
 Now, taking the $l_d$-roots in \eqref{equ:1}, we obtain
	\begin{equation}\label{equ:2}
		(V_{h_d})^{1/l_d} \ge (L_{h_d})^{1/l_d} \ge \Big[\prod_{j=1}^{h_d-1}M_{j} \Big]^{-1/l_d} \Big[\prod_{i =1}^d\mathcal{T}_i^{i(h_i-h_{i-1})} \Big]^{1/l_d}. 
	\end{equation}
Now, passing to the limit and using Corollary \ref{th:meanmod} and Lemma \ref{lem:3}, we obtain
	\begin{equation*}
		\lim_{d\to\infty} (L_{h_d})^{1/l_d} = D(K). \qedhere
	\end{equation*}
\end{proof}
As to prove Relation \eqref{eq:asympk}, we proceed as follows.  
From  Inequalities \eqref{equ7} and \eqref{equ:2}, we get  
\begin{equation*}
\Big[\prod_{j=1}^{h_d-1}M_{j} \Big] \Big[\prod_{i=1}^d\mathcal{T}_i^{i(h_i-h_{i-1})} \Big]\le \prod_{k=0}^{h_d-1} \|P_k\|_K \le V_{h_d}.
\end{equation*}
Then, we take the  $l_d-$roots and pass to the limit as $d\to\infty$ which leads to \eqref{eq:asympk}. 

\begin{remark}
    Due to inequalities \eqref{eq:lksmallerVk} and \eqref{equ7} from the above proof, we obtain a stronger version of Theorem \ref{thm:gettransdiam}, namely
    \begin{equation}
        D(K) = \lim_{k \to +\infty} L_k^{1/l_{d(k)}}, \quad \text{where } d(k) = |\kappa(k)|.
    \end{equation}
\end{remark}
\section{Intertwining pseudo Leja sequences}\label{intertwinning}
We show how to construct a pseudo Leja sequence on $\C^{p}=\C^{p_1}\times \C^{p_2}$ out of two pseudo Leja sequences on $\C^{p_i}$, $i=1,2$. To this aim, we use a natural process introduced by Calvi that generalises a method introduced by Biermann in the case $p_i=1$, see \cite{calvi2005intertwining}. As indicated in the introduction, in this section, when necessary, we use a superscript to indicate the ground spaces for the related object, so that $e_j^{(p)}$ denotes the $j$-th monomial on $\C^p$ and so on.
\par 
We will use
\begin{equation}
\varphi: \N \ni j \longmapsto \varphi(j) = \big(\varphi_1(j), \varphi_2(j)\big) \in \N^2
\end{equation}
satisfying, see the introduction for the definition of $\kappa$, 
\begin{equation}
    \kappa^{(p)}(j) = \left(\kappa^{(p_1)}\big(\varphi_1(j)\big), \kappa^{(p_2)}\big(\varphi_2(j)\big)\right) \label{def_of_phi}, \quad j\in \N.
\end{equation}
\begin{definition} Let $A = (a_j)_{j\in \N} \subset \C^{p_1}$ and $B = (b_j)_{j\in \N} \subset \C^{p_2}$ two sequences of points. The \textbf{intertwining sequence} of $A$ and $B$ is the sequence $\Omega= (\omega_j)_{j\in \N} $ of points in $\C^p$, $p=p_1+p_2$, denoted by $A\oplus B$ and defined by  
\begin{equation}
\omega_j = \big(a_{\varphi_1(j)}, b_{\varphi_2(j)})\big), \quad j\in \N. \label{def_intertwining}
\end{equation}
\end{definition}
The important point for our purpose is that we can easily express the vandermondian for an intertwining sequence $\Omega=A\oplus B$ in terms of the factor sequences $A$ and $B$. This is specified below. 
\par
We say that $X = (x_0, \dots, x_{N-1})$ is \textbf{completely ordered unisolvent} if 
\begin{equation}
    \V(x_0, \dots, x_{i-1}) \not = 0, \quad \text{for all } i=1, \dots, N.
\end{equation}
Every section $(\xi_j:\ j=0,\dots, N-1)$ of a pseudo Leja sequence $(\xi_j)_{j\in \N}$ is completely ordered unisolvent (the compact sets being assumed determining). In view of \cite[Theorem 3.3]{calvi2005intertwining}, the notion of completely ordered unisolvent array is stable by intertwining.
\par 
Given a sequence $X = (x_j)_{j \in \N}$, we set $X_j = (x_0, \dots, x_{j-1})$. The following result is a consequence of \cite[Proposition 2.5]{calvi2005intertwining}.
\begin{proposition}\label{prop:vdm_completely_ordered_unisolvent}
Let $X = (x_0, \dots, x_{N-1})$ be a completely ordered unisolvent array. We have
\begin{equation}
        \V(x_0, \dots, x_{N-1}) = \prod_{j=0}^{N-1}(e_j - L_{X_j}e_j)(x_j)
\end{equation}
where $L_{X_j}$ is the Lagrange interpolation operator at the points $X_j = (x_0, \dots, x_{j-1})$ with the convention that $L_{X_0}f =0$ and $e_j$ is the $j$-th monomial, see \eqref{equ:12}. 
\end{proposition}
\begin{proof}
The proof follows the reasoning used to prove \cite[Proposition 2.5]{calvi2005intertwining}.
\end{proof}
The next proposition follows from \cite[Theorem 4.2]{calvi2005intertwining}.
\begin{proposition} \label{prop:thm4.2_calvi_paper}
    Let $A$, $B$ and $\Omega=A\oplus B$ as above. We have
    \begin{equation}
        \Big(e_j^{(p)} - L_{\Omega_j}e_j^{(p)}\Big)(z) = \big(e_{\varphi_1(j)}^{(p_1)} - L_{A_{\varphi_1(j)}}e_{\varphi_1(j)}^{(p_1)}\big)(z^1) \cdot \big(e_{\varphi_2(j)}^{(p_2)} - L_{B_{\varphi_2(j)}}e_{\varphi_2(j)}^{(p_2)}\big)(z^2)
    \end{equation}
for all $z=(z^1, z^2) \in \C^{p_1} \times \C^{p_2}$ and for all $j \in \N$. 
\end{proposition}
We now combine Proposition \ref{prop:vdm_completely_ordered_unisolvent} and Proposition \ref{prop:thm4.2_calvi_paper} to deduce the following result. Obverse that it involves three forms of vandermondians, with variables $\omega_i$ in  $\C^p$,  $a_i$ in $\C^{p_1}$, and $b_i$ in $\C^{p_2}$.   
\begin{theorem}\label{thm:vdm_intertwining_sequence}
Let $A$, $B$ and $\Omega=A\oplus B$ as above. If, for all $j \in \N^\star$, $A_j$ and $B_j$ are completely ordered unisolvent arrays, then for all $j \in \N^\star$, we have
\begin{equation}
\V(\omega_0, \dots, \omega_{j-1},z) = P_j(z)\cdot \V(\omega_0, \dots, \omega_{j-1}), \label{equ:vdm_intertwinin_sequence}
\end{equation}
where, for $z=(z^1,z^2)$,
\begin{align}
		P_j(z) &= (e_j^{(p)} - L_{\Omega_j}e_j^{(p)})(z)\\
  &= \big(e_{\varphi_1(j)}^{(p_1)} - L_{A_{\varphi_1(j)}}e_{\varphi_1(j)}^{(p_1)}\big)(z^1) \cdot \big(e_{\varphi_2(j)}^{(p_2)} - L_{B_{\varphi_2(j)}}e_{\varphi_2(j)}^{(p_2)}\big)(z^2)
  \end{align}
  Hence, if $\varphi_i(j)>0$ for $i=1,2$, 
		\begin{equation}\label{eq:casenonzero} P_j(z)= \frac{\V(a_0, \dots, a_{\varphi_1(j)-1}, z^1)}{\V(a_0, \dots, a_{\varphi_1(j)-1})} \cdot \frac{\V(b_0, \dots, b_{\varphi_2(j)-1}, z^2)}{\V(b_0, \dots, b_{\varphi_2(j)-1})}.
	\end{equation}
 While if, say, $\varphi_1(j)>0$ and $\varphi_2(j)=0$,  
 \begin{equation}\label{eq:casezero}
     P_j(z)= \frac{\V(a_0, \dots, a_{\varphi_1(j)-1}, z^1)}{\V(a_0, \dots, a_{\varphi_1(j)-1})}. 
 \end{equation}
\end{theorem}
Now, we state the main result of this section. 
\begin{theorem} \label{thm:intertwining_pseudo_Leja_points}
    Let $K_i$, $i =1,2$, be a determining compact subset in $\C^{p_i}$. Let $A = (a_j)_{j \in \N} \subset K_1$, $B = (b_j)_{j \in \N} \subset K_2$ and $\Omega = A \oplus B$. The following assertions are equivalent:
    \begin{enumerate}
        \item \label{assertion:intertwining_pL1} $\Omega$ is a pseudo Leja sequence for $K  =K_1 \times K_2$.
	\item \label{assertion:intertwining_pL2} $A$ and $B$ are pseudo Leja sequences for $K_1$ and $K_2$ respectively. 
    \end{enumerate}
\end{theorem}
\begin{proof}
We prove $\eqref{assertion:intertwining_pL1}\implies \eqref{assertion:intertwining_pL2}$. 
Supposing that $\Omega$ is a pseudo Leja sequence for $K$ of Edrei growth $(M_j)_{j\ge 1}$. We prove that $A$ is a pseudo Leja sequence for $K_1$. Let $i\in \mathbb{N}$ and $N \in \N$ such that $ \varphi(N) = (i,0)$. From the definition of a pseudo Leja sequence, we have
\begin{equation*}
M_N|\V(\omega_0, \dots, \omega_{N-1}, \omega_N)| \ge \max_{z \in K} |\V(\omega_0, \dots, \omega_{N-1},z)|.\end{equation*}
 Applying Theorem \ref{thm:vdm_intertwining_sequence} to both sides of the above inequality, we obtain
 \begin{equation*}
     M_N|P_N(\omega_N)| \ge \max_{z \in K} |P_N(z)|.
 \end{equation*}
 Since $\omega_N=(a_i,b_0)$, in view of \eqref{eq:casezero}, the above inequality reduces to
 \begin{equation*}
     M_N |\V(a_0, \dots, a_{i-1}, a_i)| \ge \max_{z^1 \in K_1} |\V(a_0, \dots, a_{i-1}, z^1)|.
 \end{equation*}
Therefore, the sequence $A$ will satisfy the inequality property of a pseudo Leja sequence on condition that the growth sequence  
\[M'_i = M_N, \quad \varphi(N)=(i,0)\] satisfies the Edrei growth condition in \eqref{property2}, that is
    \begin{equation*}
 \lim_{d \to + \infty} \max_{h_{d-1}\le j < h_{d}}(M'_j)^{1/d} = 1,
    \end{equation*}
where $h_s=h_s^{(p_1)}$ is the dimension of the space of polynomial of degree $s$ on the ground space of $A$, that is $\C^{p_1}$. 
Let $d \ge 1$ and let \[\mathcal{M}_i =M'_i= \max_{h^{(p_1)}_{d-1} \le j < h^{(p_1)}_{d}}M'_j.\]
As previously, we consider the number $N = N(i)$ such that $\varphi(N) = (i, 0)$. Therefore, $\kappa^{(p)}(N) = (\kappa^{(p_1)}(i), 0)$ and $|\kappa^{(p)}(N)| = |\kappa^{(p_1)}(i)| = d$. It follows that $N \in \{h_{d-1}^{(p)}, \dots, h_d^{(p)}-1\}$ and hence
\[\mathcal{M}_i \leq \max_{h^{(p)}_{d-1} \le j < h^{(p)}_d} M_j.\]
Taking the $1/d$ root and passing to the limit as $d\to \infty$, we obtain the result, taking into account that $\mathcal{M}_i \ge 1$, thanks to the growth condition of pseudo Leja sequences. 
Thus, $A$ is a pseudo Leja sequence of Edrei growth $(M'_i)_{i \ge 1}$. Of course, the same reasoning shows that $B$ is also a pseudo Leja sequence for $K_2$.
\par\smallskip 	
We now prove $\eqref{assertion:intertwining_pL2} \implies \eqref{assertion:intertwining_pL1}$. We assume that $A$ is a pseudo Leja sequence for $K_1$ of Edrei growth $(M_j')_{j\ge 1}$ and $B$ is a pseudo Leja sequence for $K_2$ of Edrei growth $(M_j'')_{j\ge 1}$. We show that $\Omega=A\oplus B$ is a pseudo Leja sequence for $K$.\par
Fix $j \in \N^\star$. If $\varphi_i(j)>0$, $i=1,2$, from Theorem \ref{thm:vdm_intertwining_sequence}, we have
\begin{equation*}
    |\V(\omega_0, \dots, \omega_{j-1}, \omega_j)| = |P_j(\omega_j)|\cdot |\V(\omega_0, \dots, \omega_{j-1})|,
\end{equation*}
where
\[
    P_j(\omega_j) = \frac{\V(a_0, \dots, a_{\varphi_1(j)-1}, a_{\varphi_1(j)})}{\V(a_0, \dots, a_{\varphi_1(j)-1})} \cdot \frac{\V(b_0, \dots, b_{\varphi_2(j)-1}, b_{\varphi_2(j)})}{\V(b_0, \dots, b_{\varphi_2(j)-1})}.
\]
In view of \eqref{eq:casezero}, one readily check that the above formula remains true when $\varphi_i(j)=0$, $i=1$ or $2$. 
Using the hypothesis that the factor sequences are pseudo Leja sequences, we have
\begin{align*}
    M'_{\varphi_1(j)} M''_{\varphi_2(j)} |P_j(\omega_j)| &\ge \max_{z^1 \in K_1}  |\frac{\V(a_0, \dots, a_{\varphi_1(j)-1}, z^1)}{\V(a_0, \dots, a_{\varphi_1(j)-1})}| \cdot \max_{z^2 \in K_2} |\frac{\V(b_0, \dots, b_{\varphi_2(j)-1}, z^2)}{\V(b_0, \dots, b_{\varphi_2(j)-1})}| \\
    &= \max_{z \in K} |P_j(z)|.
\end{align*}
Therefore,  
\begin{equation*}
    M_j |\V(\omega_0, \dots, \omega_j)| \ge \max_{z \in K} |P_j(z) \V(\omega_0, \dots, \omega_{j-1})| = \max_{z \in K} |\V(\omega_0, \dots, \omega_{j-1}, z)| 
\end{equation*}
	where $M_j = M'_{\varphi_1(j)}M''_{\varphi_2(j)}$ with the convention that $M'_{0} = M''_0 = 1$. 
It remains to prove that 
 \begin{equation}\label{eq:rtp}
     \lim_{d \to + \infty} \max_{h^{(p)}_{d-1}\le j < h^{(p)}_{d}}M_j^{1/d} = 1.
 \end{equation}
 Fix $d \ge 1$. Let $j_d \in \{h_{d-1}^{(p)}, \dots,h_{d}^{(p)}-1\} $ such that \[ M_{j_d} = \max_{h^{(p)}_{d-1}\le i < h^{(p)}_{d}}M_j.\] From the definition of $\varphi$, see \eqref{eq:kappa}, we have 
\begin{align}
		& \kappa^{(p)}(j) = \big(\kappa^{(p_1)}(\varphi_1(j)), \kappa^{(p_2)}(\varphi_2(j))\big) \\
  \implies & d=|\kappa^{(p)}(j)| = |\kappa^{(p_1)}(\varphi_1(j))| + |\kappa^{(p_2)}(\varphi_2(j))|. \label{eq:1and2}
	\end{align}
 Let $\delta_i(d) = |\kappa^{(p_i)}(\varphi_i(j_d))|$ for $i=1,2$. Now we have
\[M_{j_d}^{1/d}=
\left((M'_{\varphi_1(j_d)})^{1/|\delta_1(d)|}\right)^{|\delta_1(d)|/d}
\times
\left((M''_{\varphi_2(j_d)})^{1/|\delta_2(d)|}\right)^{|\delta_2(d)|/d}.\]
We consider two cases. If say $\delta_1(d)$ is bounded then, in view of \eqref{eq:1and2}, $\delta_2(d)/d \to 1$, and \eqref{eq:rtp} follows from the fact that $\delta_1(d)/d\to 0$ and that, $B$ being a pseudo Leja sequence, $M''$ satisfies the required growth condition. Since, by \eqref{eq:1and2}, both $\delta_i(d)/d$ cannot be bounded, it remains only to the study the case for which they are unbounded hence, since they are increasing, tending to $\infty$. In that case, \eqref{eq:rtp} follows from the fact that $M'$ and $M''$ satisfies the Edrei growth condition and that $\delta_i(d)/d$ is bounded by $1$, $i=1,2$.
\end{proof} 
Theorem \ref{thm:intertwining_pseudo_Leja_points} is a generalisation of Irigoyen's recent result in \cite[Theorem 1 ]{Amadeo2021} about the intertwining of Leja sequences when $p_1 = p_2 = 1$. It is well known that the intertwining of two good univariate sequences of interpolation points (good in the sense given in the introduction) is also good for the interpolation of holomorphic functions in $\mathbb{C}^2$. This property was first proved by Siciak in \cite{siciak1962some} but can also be found in \cite{bertrand2022newton}, \cite{bialas2012pseudo} where different proofs are proposed.

\section{Computing pseudo-Leja points}\label{sec:computing}
In practice, we compute finitely many points, even the first few points, see below, of a pseudo-Leja sequences. Such a computation is based on the following observation which shows how to obtain pseudo-Leja points on $K$ as Leja points for a finite subset of $K$. 
\begin{proposition} Let $A$ a finite subset of $K\in\C^p$ that is determining for the space of polynomial of degree $\leq d$ (hence $\textbf{Card}(A) \ge h_d= h_d^{(p)}$).
We define inductively $h_d$ points as follows.
\begin{enumerate}
    \item We choose arbitrarily $\xi_0 \in A$ and,
    \item for each $d \in \mathbb{N}$, we compute recursively the points $\xi_{h_{d-1}}, \dots, \xi_{h_{d}-2}$ and $\xi_{h_{d}-1}$ from  $A$ such that 
	\begin{equation}
		|\V(\xi_0, \dots, \xi_{N-1},\xi_N)| = \max_{\xi \in A} |\V(\xi_0, \dots, \xi_{N-1}, \xi)|,
	\end{equation}
for $h_{d-1} \le N < h_d$.
\end{enumerate} 
Then the points $\xi_j$ forms the first $h_d$ points of a pseudo-Leja sequences for $K$ of (constant) Edrei growth $M_A$ where 
\[ M_A= \max\{ 1/\|P\|_A, \quad \|P\|_K=1, \deg(P)\leq d\}.\] 
These points are called \textbf{discrete Leja points} for $A$. 
\end{proposition}

\begin{proof} Observe that the finiteness of $M_A$ follows from the facts that $A$ is determining for the space of polynomial of degree $\leq d$ and that the set of polynomials of bounded degree $d$ and sup norm not bigger than $1$ on $K$ is compact. This being said, 
the claim follows immediately from the observation that for every $h_{d-1} \le N < h_{d}$, the function
 \[\xi \longmapsto |\V(\xi_1, \dots, \xi_{N-1}, \xi)|\] is a polynomial of $n$ complex variables and of degree at most $d$ so that, by definition of $M_A$,

 \[\max_{\xi\in K} |\V(\xi_1, \dots, \xi_{N-1}, \xi)| \leq M_A\max_{\xi\in A} |\V(\xi_1, \dots, \xi_{N-1}, \xi). \qedhere\]
\end{proof}

In practice, the smallest the constant $M_A$ the better the points. In general, it is not easy to select a mesh $A$ with low constant $M_A$ and low cardinality. 
The best algorithm currently available to compute discrete Leja points is provided in \cite{bos2010computing}. In short, it is as follows. If we write $ A=\{a_1, \dots, a_M\}$ ($M> h_d$) and if $\textbf{P} = \{P_1, \dots, P_{h_d}\}$ is an ordered polynomial basis then applying the standard LU factorisation with row pivoting to the following rectangular Vandermonde matrix
\begin{equation}
    \Vm(a_1, \dots, a_M;\textbf{P}) = [P_j(a_i)]_{1\le i \le M, 1\le j\le h_d}.
\end{equation} produces a permutation vector $\sigma$ whose first $h_d$ components are the indices of the first $h_d$ Leja points in A. During this process, the matrix $\Vm(a_1, \dots, a_M;\textbf{P})$ is progressively modified by operations on the rows and the pivot row is always selected from the candidates with the largest absolute value in the first column of the submatrix of interest.
Observe, due to the large dimension of the determinant involved, one can currently compute points only for low degrees, say $d=10$ for $p=2$.

Finally, we mention a (theoretical) method, based on the same idea as the above proposition, using the concept of (weakly) admissible mesh (for $K$) introduced in \cite{calvi2008uniform} which consists of a sequence of sets $(A_d)_d\in \N$, such that $A_d\subset K$ is determining for the space of polynomials of degree $\leq d$ and both $M_{A_d}$ and the cardinality of $A_d$ grows subexponentially as $d\to\infty$. The proof is immediate and is omitted. 

\begin{proposition}\label{thm:6.1} Let $K$ a determining compact set in $\C^p$. and $(A_d)_{d\in\N}$ an admissible mesh for $K$. A pseuso-Leja sequence $(\xi_N)_{N\in \N}$ for $K$ of Edrei growth $\overline{M}_N$ with
\[\overline{M}_N = M_{A_d}, \quad h_{d-1}\le N < h_{d},\]
is obtained as follows :
\begin{enumerate}
    \item  $\xi_0 \in A_0$ and, 
    \item for each $d \in \mathbb{N}$, we compute recursively the points $\xi_{h_{d-1}}, \dots, \xi_{h_{d}-2}$ and $\xi_{h_{d}-1}$ from  $\mathcal{A}_d$ such that 
	\begin{equation}
		|\V(\xi_0, \dots, \xi_{N-1},\xi_N)| = \max_{\xi \in A_d} |\V(\xi_0, \dots, \xi_{N-1}, \xi)|.
	\end{equation}
\end{enumerate}
\end{proposition}
 
\subsection*{Acknowledgement}
I would like to thank Professor Leokadia Bia{\l}as-Cie{\.z}, my research supervisor, for her patient guidance, enthusiastic encouragement and useful critiques of this work. I would also like to express my sincere gratitude to Professor Jean-Paul Calvi for his unconditional help in structuring and writing this paper.\\
Funding: This work was partially supported  by the National Science Center, Poland, grant Preludium Bis 1 N. 2019/35/O/ST1/02245.


\end{document}